\documentclass[10pt]{amsart}
\usepackage{enumerate}
\usepackage{amssymb}
\usepackage{graphicx}
\usepackage{amscd, color}
\usepackage{amsmath}
\usepackage{amsfonts}
\usepackage[english]{babel}
\usepackage{epsfig}
\usepackage{color}

\usepackage{tikz-cd}

\numberwithin{equation}{section}

\newtheorem{theorem}{Theorem}[section]
\newtheorem{proposition}[theorem]{Proposition}
\newtheorem{corollary}[theorem]{Corollary}
\newtheorem{example}[theorem]{Example}

\newtheorem{remark}[theorem]{Remark}

\newtheorem{lemma}[theorem]{Lemma}
\newtheorem{question}[theorem]{Question}
\newtheorem{definition}[theorem]{Definition}

\usepackage{tikz}
\usepackage{hyperref}
\usetikzlibrary{matrix}

\DeclareMathOperator{\ima}{Im}

\begin{document}

{\it This is an Accepted Manuscript of an article published by Taylor \& Francis Group in Africa Review on 20/03/2019, available online:\\ 
\href{https://www.tandfonline.com/doi/abs/10.2989/16073606.2019.1581298}{https://www.tandfonline.com/doi/abs/10.2989/16073606.2019.1581298} 

\vspace{.5cm}
DOI: 10.2989/16073606.2019.1581298}

\vspace{1cm}

\title[Separation axioms and dimension of asymmetric  spaces]{Separation axioms and covering dimension of asymmetric normed spaces}

\author{Victor Donju\'an and Natalia Jonard-P\'erez}

\subjclass[2010]{22A30, 46A19,   52A21, 54D10 , 54F45, 54H11, 	}

\keywords{asymmetric norm, right bounded, covering dimension, separation axioms, paratopological group}

\thanks{The first author has been supported by Conacyt grant 464720 (M\'exico). The second author has been supported by Conacyt grant 252849 (M\'exico) and by  PAPIIT grant IA104816  (UNAM, M\'exico).}

\address{Departamento de  Matem\'aticas,
Facultad de Ciencias, Universidad Nacional Aut\'onoma de M\'exico, 04510 Ciudad de M\'exico, M\'exico.}

\email{(N.\,Jonard-P\'erez) nat@ciencias.unam.mx}
\email{(V. Donju\'an) donjuan@ciencias.unam.mx}
%46A55--convex sets?
% 46B50 Compactness in Banach spaces
	%52A21  	Finite-dimensional Banach spaces (including special norms, zonoids, etc.) [See also 46Bxx]
    %	52A07  	Convex sets in topological vector spaces [See also 46A55]
  %54H11  	Topological groups [See also 22A05]
  
%22A30  	Other topological algebraic systems and their representations

%54F45 Dimension theory
%	54D10  	Lower separation axioms ($T_0$–$T_3$, etc.)

\begin{abstract}

It is well known that every asymmetric normed space is a $T_0$ paratopological group. Since  all $T_i$ axioms ($i=0, 1, 2, 3$) are pairwise non-equivalent in the class of paratopological groups, it is natural to ask if some of these axioms are equivalent in the class of asymmetric normed spaces. 
In this paper, we will consider this question. We will also show some topological properties of asymmetric normed spaces that are closely related with the axioms $T_1$ and $T_2$ (among others). In particular, we will make a remark on \cite[Theorem 13]{luis}, which states that every $T_1$ asymmetric normed space with compact closed unit ball must be finite-dimensional (as a vector space). We will show that when the asymmetric normed space is finite-dimensional, the topological structure and the covering dimension of the space can be described in terms of certain algebraic properties. In particular, we will characterize the covering dimension of every finite-dimensional asymmetric normed space. 
%%%?We characterize the covering dimension of every asymmetric normed space with finite algebraic dimension. 
 \end{abstract}

\maketitle

\section{Introduction}

Let $X$ be a real linear space and $\mathbb{R}^+$ be the set of nonnegative real numbers. An {\it asymmetric norm} $q$ on $X$ is a function $q:X\to\mathbb{R}^+$ satisfying the following conditions 

\begin{enumerate}
\item $q(ax)=aq(x)$,
\item $q(x+y)\le q(x)+q(y)$,
\item $q(x)=q(-x)=0$ implies $x=0$,
\end{enumerate}
for every $x,y\in X$ and $a\in\mathbb{R}^+$. The pair $(X,q)$ is called an {\it asymmetric normed space}.

Any asymmetric norm induces an asymmetric topology on $X$ that is generated by the asymmetric open balls $B_q(x,\varepsilon)=\{y\in X\mid q(y-x)<\varepsilon\}.$ This topology is always a $T_0$ topology on $X$ for which the vector sum on $X$  is continuous. 
Therefore,  if $(X,q)$ is an asymmetric normed space, then $(X,+)$ is a group such that the addition function is continuous. Such groups are called {\it paratopological groups}. However, in general this topology is not Hausdorff and the map $x\mapsto -x$ is not always continuous, thus $(X,q)$ fails to be a topological group. 

Let us remember that a topological group is a group endowed with a topology such that the operation and the inverse functions are continuous. It is known that every $T_0$ topological group is completely regular (see, e.g., \cite[Chapter II, Theorem 8.4]{Hewitt Ross}). Since every topological vector space is a topological group, we also have that every $T_0$ topological vector space  is completely regular (\cite[Theorem 2.2.14]{Megginson}). In the case of paratopological groups, the separation axioms $T_i$ ($i=0,1,2,3$) are pairwise non-equivalent. However, some non trivial implications may occur. For example,  T. Banakh and A. Ravsky proved in \cite{ravs} that each $T_3$ paratopological group is $T_{3\frac{1}{2}}$.

In the case of asymmetric normed spaces, some equivalences between separation axioms can occur too. 
Indeed, if $(X,q)$ is a  $T_1$  finite-dimensional asymmetric normed space then it is a normable space (and therefore it is also $T_2$). This is a strong result, proved in \cite[Corollary 11]{luis}, that we will use often.

Recall that $q^s:X\to \mathbb R^+$ denotes the norm defined through the formula
$$q^s(x)=\max\{q(x), q(-x)\}$$
(see Section~\S 2 for more details about $q^s$).

\begin{proposition}\label{norm}
Let $(X,q)$ be a finite-dimensional asymmetric normed space satisfying the axiom $T_1$. Then the norm $q^s$ induces the same topology on $X$ than that induced by $q$. In this case we say that the asymmetric normed space $(X,q)$ is normable by $q^s$.
\end{proposition}

Another interesting result related to the axiom $T_1$ was stated in \cite[Theorem 13]{luis}:

\begin{theorem}\label{finite}
The closed unit  ball $B_q[0,1]$ of a $T_1$ asymmetric normed space $(X,q)$ is $q$-compact if and only if $X$ is a  finite-dimensional linear space. 
\end{theorem}

However, we noticed that there is a gap in the `only if' part of the proof of this theorem. Indeed, the author of \cite{luis} used Proposition \ref{norm} to prove the `only if' part of Theorem \ref{finite}. But that can only be done provided that the space $X$ is finite-dimensional, which is actually what he wanted to conclude.

Nonetheless, Theorem~\ref{finite} still holds and we will give a correct proof in Section~\S 3 (Corollary~\ref{c: correccion luis}). In the same section, we explore other topological properties related to the axiom $T_1$, such as the topology of some quotient spaces.

Axiom $T_2$ was explored in \cite{hous}. In that paper, the authors characterized when an asymmetric normed space is Hausdorff in terms of the  seminorm  $\|\cdot\|_q$ defined in equation~(\ref{e:seminorm}) in Section 2 (see Proposition~\ref{haus}). In the same paper, the authors also characterized when  an asymmetric normed space  is homeomorphic to the product of a Hausdorff space times the kernel of the seminorm defined in (\ref{e:seminorm}).  Inspired by this result, in Section \S 4 we show that a similar decomposition can always be given. 
In particular, we will show that if $(X,q)$ is a finite-dimensional right bounded asymmetric normed space, then $X$ is homeomorphic to the product of a Euclidean space times the linear span of the kernel of $q$ (see Corollary~\ref{linh}).
 
 Among all the literature about asymmetric normed spaces, finite-dimensional ones are a good source of results, examples and counterexamples. Here, finite-dimensional means that the space is a finite-dimensional vector space. Namely it has a finite linearly independent spanning  set. 
If $X$ is a topological vector space,  the algebraic dimension coincides with its topological dimension (covering dimension). However, in the non symmetric case this is no longer true. In Section \S 5, we investigate the topological dimension of several asymmetric normed spaces. The main result of this section characterizes the covering dimension of every finite-dimensional asymmetric normed space (see Theorem~\ref{t:main dimension}). As a corollary of this, we get that the algebraic and covering dimension of a finite-dimensional asymmetric normed space  coincide if and only if the space is $T_1$. In other cases, the covering dimension is $0$ (and therefore the space satisfies the $T_4$ axiom) or $\infty$ (depending on whether the kernel of $q$ spans   the whole space or not).  

Finally, in Section \S 6 we investigate what other separation axioms are satisfied by the asymmetric normed spaces and we show some counterexamples.
Particularly we prove that every asymmetric space $(X,q)$ with $B_q[0,1]$ $q$-closed satisfying the axiom $T_3$ is completely regular.

\section{Preliminaries}

%If $Z$ is any linear subspace of $X$, we endow $Z$ with the asymmetric norm $q$ restricted to $Z$. %Unlike norms, we don't necessarily have $q(x)=q(-x)$. The sum function $X\times X\to X$ given by $(x,y)\mapsto x+y$ is continuous. On the other hand, the inverse function $X\mapsto X$ given by $x\mapsto -x$ might not be continuous. 
Given an asymmetric normed space $(X,q)$,  and $\epsilon>0$, let
\begin{align*}
&B_q(x,\epsilon)=\{y\in X:q(y-x)<\epsilon\},\\
&B_q[x,\epsilon]=\{y\in X:q(y-x)\le\epsilon\}.
\end{align*}
We call $B_q(x,\epsilon)$ and $B_q[x,\epsilon]$ the {\it open ball} and the {\it  closed ball} centered at $x$ of radius $\epsilon$, respectively. We endow $(X,q)$ with the topology generated by all open balls $B_q(x,\epsilon)$ ($x\in X,\epsilon >0$). With this topology, the sum function $X\times X\to X$ given by $(x,y)\mapsto x+y$ is continuous, but the inverse function $X\to X$ given by $x\mapsto -x$ is not always continuous. The multiplication by a fixed positive number, as well as the translations are homeomorphisms (hence, these spaces are homogeneous). 

Let $q^-:X\to\mathbb{R}^+$ be defined by $q^-(x)=q(-x)$. Then $q^-$ is also an asymmetric norm on $X$. If for every $x\in X$ we define 
$$q^s(x):=\max\{q(x),q^-(x)\}=\max\{q(x), q(-x)\},$$ then $q^s:X\to\mathbb{R}^+$ is a norm on $X$ and $(X,q^s)$ is the normed space associated with the asymmetric normed space $(X,q)$.

In certain cases, we will need to use the topology generated by the norm $q^{s}$, i.e., the topology  determined by the sets 
$$B_{q^{s}}(x,\varepsilon)=\{y\in X : q^s(y-x)<\varepsilon\}$$ with $x\in X$ and $\varepsilon >0$. For this reason, in order to avoid any confusion, it is important to distinguish between both topologies  at the moment of dealing with them.  Therefore,
we will say that a set $A\subset X$ is $q$-compact ($q^{s}$-compact) if it is compact in the topology generated by $q$ ($q^s$). We will similarly define the notion of $q$-open, $q$-closed, $q$-continuous ($q^{s}$-open, $q^{s}$-closed, $q^{s}$-continuous), etc.

%Therefore, for e $q, q^{-1}$ and $q^s$, so if we want to say that a set is open in $(X,q)$ we will say that it is $q$-open (same for closedness, compactness, etc.). 

It is easy to check that $(X,q)$ is a first-countable space such that the sequence $(x_n)$ converges to  $x\in X$ if and only if $$\lim_{n\to\infty}q(x_n-x)=0.$$

Since $q(x)\le q^s(x)$ for every $x\in X$, if a sequence $(x_n)$ converges to $x$ in the normed space $(X,q^s)$ then the same holds in $(X,q)$. Also, $$B_{q^s}(x,\epsilon)=B_{q}(x,\epsilon)\cap B_{q^-}(x,\epsilon).$$ 
This implies that the topology of $(X,q^s)$ is finer than both topologies induced by $q$ and $q^{-1}$, respectively. %On the other hand, since the image of any open ball $B_{q}(x,\epsilon)$ under the inverse function is $B_{q^-}(x,\epsilon)$, this function () is continuous if and only if the topologies generated by $q$, $q^-$ and $q^s$ are the same on $X$.
 Additionally, the asymmetric norm $q:(X,q^s)\to\mathbb{R}^+$ is always  continuous, and therefore the closed balls $B_q[x,\epsilon]$ are $q^s$-closed. However, in general the closed balls are not $q$-closed (see Proposition~\ref{prop bola cerrada regular}).  %We will discuss this later in Section \S 6 

For every asymmetric normed space $(X,q)$, let us consider the map $\|\cdot\|_q:X\to \mathbb R^+$ defined by
\begin{equation}\label{e:seminorm}
\left\|{x}\right\|_q=\inf\{q(y)+q(y-x):y\in X\}.
\end{equation}

It was proved in \cite{hous} that $\left\|{x}\right\|_q$ is the greatest (symmetric) seminorm on $X$ such that $\left\|{x}\right\|_q\le q(x)$ for every $x\in X$. There, the following result was stated.

\begin{proposition}\label{haus}For any asymmetric normed space $(X,q)$ we have:
\begin{enumerate}[\rm(1)]
\item $X$ is $T_1$ if and only if $q(x)>0$ for every $x\in X\setminus\{0\}$.
\item $X$ is $T_2$ if and only if $\left\|{x}\right\|_q>0$ for every $x\in X\setminus\{0\}$.
\end{enumerate}
\end{proposition}

For any asymmetric normed space $(X,q)$, we consider the sets
$$\theta_q=\{x\in X:q(x)=0\},$$
$$\ker\left\|{\cdot}\right\|_q=\{x\in X: \left\|{x}\right\|_q=0\}.$$

It is easy to see that $\theta_q$ is a convex cone (i.e., it is closed under sums and positive scalar multiplication), and $\ker\left\|{\cdot}\right\|_q$ is  a linear subspace of $X$. Thus from  Proposition \ref{haus} we get that
\begin{itemize}
\item [(A1)] $\theta_q=\{0\}$ if and only if $X$ is $T_1$.
\item [(A2)] $\ker\left\|{\cdot}\right\|_q=\{0\}$ if and only if $X$ is $T_2$.

\end{itemize}

There is some disagreement in the mathematical literature about the definition of the separation axioms $T_3$, $T_{3\frac{1}{2}}$ and $T_4$. To avoid any confusion we will follow \cite{Hewitt Ross} (see also \cite{counterexamples}). Namely, a \textit{regular} space will be a space satisfying the axioms $T_1$ and $T_3$. Analogously by a  \textit{completely regular}   (\textit{normal}) space we mean a  space satisfying the axioms $T_1$ and $T_{3\frac{1}{2}}$ ($T_1$ and $T_{4}$, respectively). It's worth noting that those definitions may not be the most commonly used (see, for example  \cite{engel}).

It is widely known that any topological space satisfying the axioms $T_{3}$ and $T_0$ is a Hausdorff space (hence, regular). For the sake of completeness we include a short proof of this statement.

\begin{proposition}\label{prop cregular}
Let $X$ be a topological space satisfying the axioms $T_{3}$ and $T_0$. Then $X$ is $T_2$.
\end{proposition}

\begin{proof}
Let $x,y\in X$ with $x\neq y$. Since $X$ is $T_0$, without loss of generality there exists an open set $U$ such that $x\in U$ and $y\notin U$. Using the fact that $X$ is $T_3$, we can find two disjoint open sets, $V$ and $W$ such that $x\in V$ and $X\setminus U\subset W$. Since $y\in X\setminus U$, we can conclude that $V$ and $W$ separate $x$ and $y$. 
\end{proof}

Since every asymmetric normed space is $T_0$, from Proposition \ref{prop cregular} we directly infer the following.
\begin{corollary}\label{corol t3 implica t2} Every $T_3$ asymmetric normed space is Hausdorff.
\end{corollary}

% \begin{proposition}\label{prop cregular}
% Let $X$ be a topological space satisfying the axioms $T_{3\frac{1}{2}}$ and $T_0$. Then $X$ is $T_2$.
% \end{proposition}

% \begin{proof}
% Let $x,y\in X$ with $x\neq y$. Since $X$ is $T_0$, without loss of generality there exists an open set $U$ such that $x\in U$ and $y\notin U$. Let $f:X\to [0,1]$ be a continuous function such that $f(x)=0$ and $f(X\setminus U)=\{1\}$. Then $A=f^{-1}([0,1/2))$ and $B=f^{-1}((1/2,1])$ are disjoint open sets in $X$ such that  $x\in A$ and $y\in B$.
% \end{proof}

On an asymmetric normed space, there is a simple condition which implies the separation axiom $T_3$.

\begin{proposition}\label{prop bola cerrada regular}
Let $(X,q)$ be an asymmetric normed space. If the closed ball $B_q[0,1]$ is a $q$-closed set in $X$, then $X$ is $T_3$.
\end{proposition}
\begin{proof}
If $B_q[0,1]$ is $q$-closed, it is easy to see that all closed balls are $q$-closed too. Let $x\in X$ and $\epsilon>0$.  Then clearly $$x\in B_q(x,\epsilon/2)\subseteq{\overline{B_q(x,\epsilon/2)}}\subseteq B_q[x,\varepsilon/2]\subseteq B_q(x,\epsilon).$$
Since $B_q(x,\epsilon)$ was an arbitrary open basic neighborhood of $x$, this shows that $X$ is $T_3$.
\end{proof}

The following lemma will be used several times. The reader can find its proof in  \cite[Lemma 3.6]{Jonard Sanchez}. 

\begin{lemma}\label{rayo}
Let $(X,q)$ be an asymmetric normed space and let $x,y\in X$. If an open ball $B_q(z,\epsilon)$ contains the ray $\{tx+y:t\ge 0\}$ then $x\in\theta_q$.
\end{lemma}

From Lemma~\ref{rayo} we get that if the space is $T_1$, then the open (and closed) balls do not  contain non-trivial rays.

% For instance, consider $q:\mathbb{R}\to\mathbb{R}^+$ given by $$q(x)=x^+,$$where $x\in\mathbb{R}$ and $x^+$ is the positive part of $x$, namely $\max\{x,0\}$. Then $q$ is an asymmetric norm on $\mathbb{R}$ and $\theta_q$ is the set of nonpositive real numbers, hence $(\mathbb{R},q)$ is not $T_1$. Also, the open ball $B_q(x,\epsilon)$ is the interval $(-\infty, x+\epsilon)$ and the closed ball $B_q[x,\epsilon]$ is $(-\infty,x+\epsilon]$. It follows that the closed ball is not a closed set in $(\mathbb{R},q)$ since its complement $(x+\epsilon,+\infty)$ does not contain any open ball.

The statements of the following lemma will be useful. 

\begin{lemma}\label{suma}
Let $(X,q)$ be an asymmetric normed space.  
\begin{enumerate}[\rm(1)]
\item Let $x_1,\ldots,x_n\in\theta_q$ with $x_n\neq 0$. Then $x_1+\ldots+x_n\neq 0$.
\item For every $x,y\in X$, we have that $q(x+y)\ge q(x)-q(-y)$.
\end{enumerate}
\end{lemma} 

\begin{proof}
Suppose that $x_1+\ldots+x_n=0 $, so $x_1+\ldots+x_{n-1}=-x_n$. Therefore $-x_n\in\theta_q$, because $\theta_q$ is closed under addition. From this we conclude that $q(x_n)=q(-x_n)=0$, contradicting the hypothesis  $x_n\neq 0$.

Now, let $x,y\in X$. By the triangular inequality, $$q(x)=q(x+y-y)\le q(x+y)+q(-y).$$ It follows then that $q(x)-q(-y)\le q(x+y)$.
\end{proof}

Given a vector space $X$ and $M\subset X$, we will denote by $\left <M\right>$ the linear span of $M$.

\section{$T_1$ separation axiom and quotient of  asymmetric normed spaces}

As a consequence of Proposition~\ref{norm}, we have the following lemma.

\begin{lemma}\label{zeta}
Let $(X,q)$ be an asymmetric normed space and let $Y=\left<{\theta_q}\right>$, the subspace spanned by $\theta_q$. If $Z$ is a subspace such that $Y\cap Z=\{0\}$, then the following statements hold:\
\begin{enumerate}[\rm(1)]
\item $Z$ is $T_1$.
\item If $Z$ is a  finite-dimensional linear space then $Z$ is normable by $q^s$.
\end{enumerate}

\end{lemma}

\begin{proof}
(1) Since $Y\cap Z=\{0\}$, then $q(x)>0$ for every $x\in Z\setminus\{0\}$. By Proposition \ref{haus}, $(Z,q)$ is $T_1$.
If $Z$ is finite-dimensional, (2) follows directly from Proposition \ref{norm}. 
\end{proof}

%\begin{lemma}\label{rayo}
%Let $(X,q)$ be a $T_1$ asymmetric normed space. If $B_q[0,1]$ contains the set $\{tx:t\ge 0\}$ for some $x\in X$ then $x=0$.
%\end{lemma} 

%\begin{proof}
%For every $t\ge 0$ we have $q(tx)=tq(x)\le 1$. If $x\neq 0$ then $q(x)>0$ by Proposition \ref{haus}, then choosing $t=2/q(x)$ we obtain the inequality $2\le 1$. Therefore $x=0$.
%\end{proof}

We say that two asymmetric norms $q$ and $p$ on $X$ are {\it equivalent} if there exist two numbers $M,N>0$ such that for every $x\in X$, $$Mp(x)\le q(x)\le Np(x).$$ 

It follows that $p$ and $q$ are equivalent if and only if there are $M,N>0$ such that $$B_p(0,M)\subseteq B_q(0,1)\subseteq B_p(0,N).$$ 

Also, if $q$ and $p$ are equivalent then $(X,q)$ and $(X,p)$ have the same topology. Since we always have $B_{q^s}(0,1)\subseteq B_q(0,1)$, then in order for $q$ and $q^s$ to be equivalent it suffices that $B_q(0,1)$ is a $q^s$-bounded set. Clearly, these equivalences remain  true  if we use closed balls instead of open balls. 

\begin{lemma}\label{acota}
Let $(X,q)$ be an asymmetric normed space. If $B_q(0,1)$ (or $B_q[0,1]$) is a $q^s$-bounded set then $q$ and $q^s$ are equivalent. In particular $(X,q)$ is normable by $q^s$. 
\end{lemma}

Now we prove the `only if' part of Theorem \ref{finite}.

\begin{corollary}\label{c: correccion luis}
Let $(X,q)$ be a $T_1$ asymmetric normed space. If $B_q[0,1]$ is $q$-compact then 
\begin{enumerate}[\rm(1)]
\item $B_q[0,1]$ is $q^s$-bounded,
\item $X$ is normable by $q^s$ and
\item $X$ is finite-dimensional. 
\end{enumerate} 
\end{corollary}
\begin{proof}
Let us suppose that $B_q[0,1]$ is a $q$-compact set in the space $(X,q)$. Then the sphere $$S_q=B_q[0,1]\setminus B_q(0,1)=\{x\in X:q(x)=1\}$$ is  $q$-closed  in $B_q[0,1]$ (with respect to the subspace topology on $B_q[0,1]$). Since $B_q[0,1]$ is $q$-compact, we conclude that $S_q$ is also $q$-compact in $B_q[0,1]$, and therefore $S_q$ is $q$-compact in the whole space $X$. 

First, we will prove (1). Otherwise, if $B_q[0,1]$ is not $q^s$-bounded, there exists a sequence $(x_n)\subseteq B_q[0,1]$ such that $$\lim_{n\to\infty}q^s(x_n)=\infty.$$ Since $q(x_n)\le 1$, we can assume that $q(x_n)<q^s(x_n)$ and therefore $q^s(x_n)=q(-x_n)$ for all $n\in\mathbb{N}$. For each $n\in\mathbb{N}$ consider the point $$y_n=\frac{x_n}{q(-x_n)}.$$
Since $q(y_n)=q(-x_n)^{-1}q(x_n)\leq 1$, we have that $(y_n)\subset B_q[0,1]$. Thus, by 
the compactness of $B_q[0,1]$ we can assume that the sequence $(y_n)$ converges to some $y\in B_q[0,1]$. Furthermore, $q(-y_n)=1$ and then $(-y_n)\subseteq S_q$. Since $S_q$ is $q$-compact, we can suppose without loss of generality that $(-y_n)$ converges to some point $z\in S_q$ (in particular $z\neq 0$). By the continuity of the sum function, the trivial sequence $(y_n+(-y_n))$ converges to $y+z$. Hence $y+z\in{\overline{\{0\}}}^q=\{0\}$ because $X$ is $T_1$, so $y=-z\neq 0$. This shows that $(-y_n)$ converges to $-y$, and then $$\lim_{n\to\infty}q^s(y_n-y)=\lim_{n\to\infty}\max\{q(y_n-y),q(-y_n+y)\}=0.$$ 

This means that  $(y_n)$ converges to $y$ in the normed space $(X,q^s)$. Now, let $t$ be any nonnegative real number. Pick $n\in\mathbb{N}$ large enough such that 
$t/q(-x_n)<1$. Then, since $x_n\in B_q[0,1]$ we get that $q(ty_n)\le 1$, thus $ty_n\in B_q[0,1]$. So $ty\in B_q[0,1]$ because $B_q[0,1]$ is $q^s$-closed. By Lemma \ref{rayo} we must have $y=0$, a contradiction.
We conclude that $B_q[0,1]$ must be a $q^s$-bounded set.

 By Lemma \ref{acota} the spaces $(X,q)$ and $(X,q^s)$ have the same topology and therefore (2) holds. Finally, since the topology generated by $q^s$ coincides with the one generated by $q$, we thus have that the $q^s$-closed set $B_{q^s}[0,1]$  is also $q$-closed. Then $B_{q^s}[0,1]$ is a $q$-closed subset of the $q$-compact set $B_q[0,1]$ and therefore it is also $q$-compact. This shows that $X$ is  a locally compact normed space and so it must be finite-dimensional.
\end{proof}

\begin{remark}
Statement (3) can also be deduced from \cite[Proposition 2.4.14]{cobzas}, where the result is proved for asymmetric locally convex spaces. 
\end{remark}

\begin{remark}
Statement (1) in Corollary~\ref{c: correccion luis} does not hold without the separation axiom $T_1$. Indeed, consider the asymmetric normed space $(\mathbb R, q)$ where $q(x)=x^+$ for every $x\in \mathbb R$. In this case $q^s(x)=|x|$ and the unit ball $B_q[0,1]=(-\infty, 1]$ is $q$ compact but not $q^s$ bounded.   
\end{remark}

\subsection{Quotient subspaces}

Let $(X,q)$ be an asymmetric normed space. If $Y$ is a linear subspace of $X$, let us remember that the linear space $X/Y$ of all cosets $x+Y$ is called the quotient space of $X$ by $Y$. The asymmetric norm $q$ induces the function $w_Y:X/Y\to\mathbb{R}^+$ given by $$w_Y(x+Y)=\inf\{q(x+y):y\in Y\}.$$

Although $w_Y$ is nonnegative, positively homogeneous and satisfies the triangular inequality, it might not be an asymmetric norm. Rather, it is an asymmetric {\it seminorm} (see \cite[Chapter 1]{cobzas} for more information about asymmetric seminorms). 
We endow $X/Y$ with the topology induced by the asymmetric seminorm $w_Y$, in the same way we do for asymmetric norms.

If $q$ is a (symmetric) norm in $X$, then $w_Y$ is the usual quotient seminorm for linear subspaces. In this case, the topology generated by $w_Y$ coincides with the quotient topology induced by the canonical map $\pi:X\to X/Y$ (see e.g. \cite[Chapter III, Theorem 4.2]{Conway}).

In the asymmetric case, the same situation holds. Indeed, it is not difficult to prove that $\pi(B_q(x_0,\varepsilon))=B_{w_Y}(x_0+Y,\varepsilon)$ and therefore the map $\pi:(X,q)\to (X/Y,w_Y)$ is a continuous,  open and onto map (hence, it is an identification) (see also \cite[Proposition 3.1]{Carmen}).

 If $X$ is a normed space, it is well known that  $w_Y$ is a norm on $X/Y$ if and only if $Y$ is closed. In \cite[Proposition 8]{Valero} it is proved that if $Y$ is a $q$-closed linear subspace of the asymmetric normed space $(X,q)$, then $w_Y$ defines  an asymmetric norm on $X/Y$.
This condition is not necessary, as we can see in \cite[Proposition 3.2]{Carmen}, where the authors proved that $w_Y$ is an asymmetric norm on $X/Y$ if $Y$ is a $(q, q^{-1})$-closed subset (i.e., if $y, -y\in\overline{Y}$ then $y\in Y$).

On the other hand, the $q$-closedness of $Y$ characterizes the $T_1$ axiom on the quotient space This was proved in \cite[Proposition 3.3]{Carmen} and we enounce it in the following proposition.

 \begin{proposition}\label{coc1}
Let $(X,q)$ be an asymmetric normed space and $Y$ be a linear subspace of $X$. Then the quotient $X/Y$ is $T_1$ if and only if $Y$ is $q$-closed.
\end{proposition}

\begin{proposition}\label{coc}
Let $(X,q)$ be an asymmetric normed space and $Y$ be a linear subspace of $X$.
If $X/Y$ is $T_1$  then $\overline{\left<{\theta_q}\right>}\subseteq Y$, where $\left<{\theta_q}\right>$ denotes the linear subspace of $X$ spanned by $\theta_q$, and $\overline{\left<{\theta_q}\right>}$ its $q$-closure.
\end{proposition} 

\begin{proof}
Suppose that $X/Y$ is $T_1$ and take $x\in\theta_q$, i.e., $q(x)=0$. Then $w_Y(x+Y)=0$, which implies $x\in Y$. Hence $\left<{\theta_q}\right>\subseteq Y$. Since $Y$ is $q$-closed by Proposition~\ref{coc1}, we get that  $\overline{\left<{\theta_q}\right>}\subseteq Y$.
\end{proof}

To finish this section, we will make a final remark regarding $T_1$ and $T_2$ asymmetric normed spaces.
By Propositions \ref{coc1} and \ref{coc}, if $Y$ is a $q$-closed linear subspace of an asymmetric normed space $(X,q)$, then $Y$ contains $\overline{\left<{\theta_q}\right>}$. However, as we will show in Example \ref{parabola}, $\overline{\left<{\theta_q}\right>}$ might not be a linear subspace. 
Let $Y_{\theta}$ be defined as the smallest $q$-closed subspace of $X$ containing $\theta_q$. Therefore, $X/Y_{\theta}$ is a $T_1$ space and $Y_{\theta}$ is the smallest subspace such that the quotient is $T_1$.

On the other hand, in \cite{hous} it was proved that $X/\ker\left\|{\cdot}\right\|_q$ is a $T_2$ space (in particular, $Y_{\theta}\subseteq\ker\left\|{\cdot}\right\|_q$). Let us show that $\ker\left\|{\cdot}\right\|_q$ is the smallest subspace such that the quotient is $T_2$.

\begin{proposition}\label{prop cociente t2}
Let $(X,q)$ be an asymmetric normed space and $Y$ be a linear subspace of $X$. If $X/Y$ is $T_2$ then $\ker\left\|{\cdot}\right\|_q\subseteq Y$.
\end{proposition}

\begin{proof}
Let us show that the corresponding seminorm $\left\|{\cdot}\right\|_{w_Y}$ satisfies the following inequality
$$\left\|{x+Y}\right\|_{w_Y}\le\left\|{x}\right\|_{q},\;\text{ for all }x\in X.$$

By definition, $$\left\|{x+Y}\right\|_{w_Y}=\inf\{w_Y(x_0-x+Y)+w_Y(x_0+Y):x_0\in X\}.$$
Since $w_Y(x_0-x+Y)\le q(x_0-x)$ and $w_Y(x_0+Y)\le q(x_0)$ for all $x_0\in X$, we obtain $$\left\|{x+Y}\right\|_{w_Y}\le\inf\{q(x_0-x)+q(x_0):x_0\in X\}=\left\|{x}\right\|_{_q}.$$

This proves the inequality. Therefore, if $x\in\ker\left\|{\cdot}\right\|_q$ then $\left\|{x+Y}\right\|_{w_Y}=0$. Since $X/Y$ is $T_2$, then $x+Y=Y$, hence $x\in Y$.
\end{proof}

\section{Product structure of asymmetric normed spaces}

In the case of a normed space, say $X$, if $Y$ is a closed linear subspace and $X/Y$ is finite-dimensional, then $X$ is homeomorphic to $Y\times Z$, for any subspace $Z$ complementary to $Y$ (see \cite[Chapter III]{bessa}). We will show that the same result holds in asymmetric normed spaces. With this, we can give a natural decomposition of the so called right-bounded asymmetric normed spaces. 

In \cite[Theorem 14]{hous}, the authors proved that any asymmetric normed space $(X,q)$ is linearly homeomorphic to $\ker\left\|{\cdot}\right\|_q\times (X/\ker\left\|{\cdot}\right\|_q)$ if and only if $\ker\left\|{\cdot}\right\|_q$ is complemented. Inspired by the proof of this theorem, we will show  that in fact $X$ is linearly homeomorphic to $Y\times (X/Y)$ (with a certain condition) if and only if $Y$ is complemented. Let us recall some definitions.

\begin{definition} Let $X$ be a linear space and $Y$ be a linear subspace of $X$.
\begin{enumerate}[\rm(1)]
\item A {\it projection} of the subspace $Y$ is a linear function $Q:X\to Y$ such that $Q(y)=y$ for every $y\in Y$.
\item If $Z$ is a subspace of $X$, we say that $Z$ is an algebraic complement of $Y$ if $X=Y\oplus Z$. Namely, $Y\cap Z=\{0\}$ and $X=Y+Z$.
\item If $(X,q)$ is an asymmetric normed space, the subspace $Y$ is called topologically complemented if there exists a projection $Q:X\to Y$ such that $Q$ and $I-Q$ are continuous functions in $(X,q)$, where $I:X\to X$ is the identity\footnote{Unlike the normed case, in asymmetrically normed spaces the continuity of $Q$ does not necessarily imply the continuity of $I-Q$.}. The subspace $Z=\ima (I-Q)$ will be called a topological complement of $Y$.
\end{enumerate}
\end{definition}

It is well-known that if $Q:X\to Y$ is a projection then $X=\ima Q\oplus\ker Q$ (see for instance \cite{friedberg} or \cite{Conway}).

%If $Q:X\to Y$ is a projection, then for every $x\in X$, $$x=(x-Q(x))+Q(x)\in\ker Q+\ima Q.$$
%Furthermore, if $z\in\ker Q\cap\ima Q$, then $z=Q(x)$ for some $x$ and $$0=Q(z)=Q(Q(x))=Q(x)=z,$$ hence $X=\ker Q\oplus\ima Q$, so $\ker Q$ and $\ima Q$ are complementary to each other. Also, it is easy to see that $Y=\ima Q=\ker (I-Q)$ and $\ker Q=\ima (I-Q)$. Since the function $Q_1=I-Q$ is a projection, then the subspace $Z=\ima(I-Q)$ is (topologically) complemented if and only if $Y$ is (topologically) complemented. Clearly in this case $Y$ and $Z$ are complementary to each other.

%In particular, if $X=Y\oplus{Z}$, the projection onto $Y$ along $Z$ is the function $Q:X\to Y$ given by $Q(x)=y_x$, where $y_x$ is the only element in $Y$ such that $x=y_x+z$ with $z\in Z$. It is easy to check that the function $\psi:X/Y\to Z$ defined by $\psi(x+Y)=(I-Q)(x)$ is a linear isomorphism such that $I-Q=\psi\circ\pi$, where $\pi:X\to X/Y$ denotes, as always, the quotient map. 

On the other hand, it was proved in \cite[Lemma 10]{hous} that a linear function between asymmetric normed spaces, say $f:(X,q)\to (X_1,p)$, is continuous if and only if there exists $K>0$ such that $p(f(x))\le K q(x)$ for all $x\in X$. Hence, a subspace $Y$ of $(X,q)$ is complemented if and only if there exist a projection $Q:X\to Y$ and $K>0$ such that for every $x\in X$, $$\max\{q(Q(x)),q(x-Q(x))\}\le Kq(x).$$

In what follows, if $(X,q)$ and $(Y,p)$ are asymmetric normed spaces, we will assume that the product $X\times Y$ is endowed with the asymmetric norm $$q^*(x,y)=\max\{q(x),p(y)\}.$$
It is easily seen that $q^*$ induces the product topology on $X\times Y$.

\begin{theorem}\label{compl}

Let $(X,q)$ be an asymmetric normed space and $Y$ and $Z$ linear subspaces of $X$ such that $X=Y\oplus Z$, with $P_Y$ the projection of $X$ onto $Y$ along $Z$. Define the maps $\phi:Y\times Z\to X$ and $\psi:X/Y\to Z$ given by $\phi(y,z)=y+z$ and $\psi(x+Y)=(I-P_Y)(x)$. Then the following statements are equivalent.
\begin{enumerate}[\rm(1)]
\item $P_Y$ and $I-P_Y$ are continuous.
\item $\phi$ is a homeomorphism.
\end{enumerate}
In this case $\psi:X/Y\to Z$ is a homeomorphism as well.
\end{theorem}

\begin{proof}
Let us suppose that (1) is true. Clearly $\phi$ is a linear isomorphism whose inverse is given by $\phi^{-1}(x)=(P_Y(x),(I-P_Y)(x))$. Since $P_Y$ and $I-P_Y$ are continuous then $\phi^{-1}$ is continuous too. On the other hand, for $(y,z)\in Y\times Z$ we have $$q(\phi(y,z))=q(y+z)\le q(y)+q(z)\le 2q^*(y,z),$$ and hence $\phi$ is continuous.

Suppose now that (2) is true. Since $\phi^{-1}$ is continuous there exists $C>0$ such that $$q^*(P_Y(x),(I-P_Y)(x))=q^*(\phi^{-1}(x))\le Cq(x).$$
This proves that both $P_Y$ and $I-P_Y$ are continuous, and since $Z=\ima(I-P_Y)$ then $Z$ is a topological complement of $Y$.

Finally, let us show that in this case $\psi$ is a homeomorphism as well. The inverse function of $\psi$ is given by $\psi^{-1}(z)=z+Y.$ For any $z+Y\in X/Y$ we have $$w_Y(z+Y)=\inf\{q(z+y):y\in Y\}\le q(z+P_Y(-z))=q(\psi(z+Y)),$$ hence $\psi^{-1}$ is continuous.

The continuity of $\psi$ follows from \cite[Proposition 2.4.2]{engel}.\end{proof}

As we mentioned earlier, it is well known that in (symmetric) normed spaces, $X$ admits a decomposition of the form $Y\times Z$ if $X/Y$ is finite-dimensional and $Y$ is closed. We will show that such decomposition holds in asymmetric normed spaces too. According to Theorem \ref{compl}, it is enough to prove that $Y$ is topologically complemented. 

% Let us first prove the following lemma.

% \begin{lemma}\label{continua}
% Let $X,Y$ be asymmetric normed spaces and $f:X\to Y$. If $f:X\to (Y,q^s)$ is continuous, then so is $f:X\to (Y,q)$.
% \end{lemma}
% \begin{proof}
% Since any $q$-open set is $q^s$-open, then the continuity of $f:X\to (Y,q^s)$ implies that $f^{-1}(A)$ is open in $X$ for every $q$-open set $A$.
% \end{proof} 

\begin{proposition}\label{decomp}
Let $(X,q)$ be an asymmetric normed space and $Y$ be a $q$-closed subspace of $X$. If $X=Y\oplus Z$ and $Z$ is finite-dimensional, then  $(Z,q)$ is linearly homeomorphic to $X/Y$ and $Y$ and $Z$ are (topologically)-complementary to each other. 
\end{proposition}

\begin{proof}
Let $Q:X\to Y$ be the projection onto $Y$ along $Z$. Since $\psi:X/Y\to Z$ given by $\psi(x+Y)=(I-Q)(x)$ is a linear isomorphism, $X/Y$ is finite-dimensional too. Moreover, since $Y$ is $q$-closed, then by Proposition \ref{coc1} $X/Y$ is a $T_1$ asymmetric normed space. By Theorem \ref{norm} $X/Y$ is normable. This implies that the map between normed spaces $-\psi:X/Y\to (Z,q^s)$ is continuous and therefore it remains continuous when $Z$ has the weaker topology inherited by $q$.

If $\pi:X\to X/Y$ is the quotient map, then for every $x\in X,$ $$-\psi(\pi(x))=-\psi(x+Y)=Q(x)-x.$$

Hence $Q=I+(-\psi)\circ\pi$ and $I-Q=\psi\circ\pi$ are continuous projections, so $Y$ is topologically complemented (and so is $Z$). Finally, Theorem \ref{compl} implies that $(Z,q)$ is linearly homeomorphic to $X/Y$. \end{proof}

If $X$ is a normed space, and $Q:X\to Y$ is a continuous projection, then we always have that $Y=\ker(I-Q)$ and $Z=\ker Q$ are necessarily  closed linear subspaces. This situation changes in the case of asymmetric normed spaces, as we can see in the following example.

\begin{example}
Let $X=\mathbb{R}^2$ and consider the asymmetric norm $q(x,y)=\max\{x^+,|y|\}$. It is easy to see that $Y=\{(x,y):y=0\}$ is $q$-closed. Since $Z=\{(x,y):x=0\}$ is finite-dimensional and $X=Y\oplus Z$, then, by Proposition \ref{decomp}, $Y$ and $Z$ are topologically complemented to each other.  However, $Z$ is not $q$-closed because any point $(x,y)$ with $x>0$ is in the $q$-closure of $Z$. 
\end{example}

Despite this situation, the subspace $Z$ may have nice properties, as we can see in the following.

\begin{corollary}\label{corol}
Let $(X,q)$ be an asymmetric normed space, and $Y$ be a $q$-closed subspace of $X$. If $X=Y\oplus Z$ and $Z$ is finite-dimensional, then $X$ is linearly homeomorphic to $Y\times Z$ and $Z$ is a $T_1$ subspace of $(X,q)$.
\end{corollary}

\begin{proof}
By Proposition \ref{decomp}, $Y$ is complemented and $Z$ is linearly homeomorphic to $X/Y$. By Theorem \ref{compl}, $X$ is linearly homeomorphic to $Y\times Z$. The space $X/Y$ is $T_1$ because $Y$ is $q$-closed, by Proposition \ref{coc1}. Hence $Z$ is $T_1$.
\end{proof}

Remember that by Proposition \ref{coc}, any $q$-closed subspace must contain $\overline{\left<{\theta_q}\right>}$.  Hence, if $Y=\overline{\left<{\theta_q}\right>}$ is a linear subspace of $(X,q)$, then it is the smallest subspace for which the decomposition in Corollary \ref{corol} holds. However, it turns out that $\overline{\left<{\theta_q}\right>}$ might not be a linear subspace of $X$. Besides, even if it is a subspace, it could happen that $X=\overline{\left<{\theta_q}\right>}$ and then such decomposition would be trivial. We illustrate these situations with the following example.

\begin{example}\label{parabola}
(1) Let $X=\mathbb{R}^2$. In \cite[Example 2.7]{Conradie 2}, an asymmetric norm $q$ was defined  by $$q(x,y)=\frac{1}{2}\left(-y+\sqrt{4x^2+y^2}\right).$$ This asymmetric norm satisfies $$B_q(0,1)=\{(x,y):y>x^2-1\},$$ and $\theta_q=\{(0,y):y\ge 0\}$. Thus $\left<{\theta_q}\right>=\{(0,y):y\in\mathbb{R}\}$. The $q$-closure $\overline{\left<{\theta_q}\right>}$ is the whole space $\mathbb{R}^2$, because all nonempty $q$-open sets intersect $\theta_q$.

(2) Consider $X$ and $q$ given as in $(1)$, and let $p:X\to\mathbb{R}^+$ defined by $p(x,y)=\max\{|x|,y^-\}$, where $y^-=\max\{-y,0\}$. Let $r:X\to\mathbb{R}^+$ be 
\begin{equation*}
   r(x,y) = \begin{cases}
               p(x,y)          & x\le 0\\
               q(x,y)          & x\ge 0
           \end{cases}
\end{equation*}
Then $r$ is an asymmetric norm on $X$ and we have $\left<{\theta_{r}}\right>=\{(0,y):y\in\mathbb{R}\}$, while its $r$-closure is $\{(x,y):x\le 0\}$ which is not a linear subspace of $X$.
\end{example}

To finish this section, we will show that there is a class of asymmetric normed spaces such that $\overline{\left<{\theta_q}\right>}=\left<{\theta_q}\right>$. Therefore, in this case the $q$-closure of $\left<{\theta_q}\right>$ is indeed a linear subspace of $X$.

In \cite[Definition 16]{luis},  the class of right-bounded asymmetric normed spaces was defined. An asymmetric normed space $(X,q)$ is called {\it right-bounded} if there exists $r>0$ such that $$B_q(0,1)\subseteq B_{q^s}(0,r)+\theta _q.$$

\begin{proposition}\label{derecho}
Let $(X,q)$ be a right-bounded asymmetric normed space with the  following  two properties: 
\begin{enumerate}[\rm(1)]
\item $\left<{\theta_q}\right>$ is $q^s$-closed.
\item There exists a norm $\left\|{\cdot}\right\|$ equivalent to $q^s$ such that $(X,\left\|{\cdot}\right\|)$ is a Hilbert space.
\end{enumerate}
Then $\left<{\theta_q}\right>$ is a $q$-closed linear subspace of $X$. 

In particular, if $(X,q)$ is a right-bounded finite-dimensional asymmetric normed space, then $\left<{\theta_q}\right>$ is a $q$-closed linear subspace of $X$.
\end{proposition}

\begin{proof}
Let us call $Y=\left<{\theta_q}\right>$. Since $Y$ is $q^s$-closed, it is also closed in the Hilbert space $(X,\left\|{\cdot}\right\|)$. Therefore $X=Y\oplus{Y}^{\perp}$, where $Y^{\perp}$ is the orthogonal complement of $Y$ (see, e.g. \cite[Theorem 3.3-4]{krey}). Let $N>0$ be such that $$N\left\|{x}\right\|\le q^s(x),$$for all $x\in X$. If $y\in Y$ and $z\in Y^{\perp}$, we know that $$\left\|{y+z}\right\|^2=\left\|{y}\right\|^2+\left\|{z}\right\|^2.$$

Hence, for any $y\in Y$ and $z\in Y^\perp$, 

\begin{equation}\label{d:norma y qs}
N\left\|{z}\right\|\le N\left\|{y+z}\right\|\le q^s(y+z).
\end{equation}

To prove that $Y$ is $q$-closed, let $(y_n)\subseteq Y$ be a sequence that converges to some $x\in X$. This means that $q(y_n-x)$ converges to zero. So, without loss of generality, we can suppose that $q(y_n-x)\le 1/n$ for all $n\in\mathbb{N}$. Since $X$ is right-bounded, there exists $r>0$ such that $$B_q(0,1)\subseteq B_{q^s}(0,r)+\theta _q.$$

Therefore $y_n-x\in B_{q^s}(0,r/n)+\theta _q$. For each $n\in\mathbb{N}$, pick $v_n\in\theta _q$ with the property that $y_n-x-v_n\in B_{q^s}(0,r/n)$. Let $y\in Y$ and $z\in Y^{\perp}$ be such that $x=y+z$. Then    
$$N\left\|{z}\right\|\le q^s(y_n-v_n-y-z)<r/n,\; \text{ for every }n\in\mathbb{N}.$$ 
This implies that $z=0$ and hence $x\in Y$, as desired.
\end{proof}

Condition (2) in Proposition \ref{derecho} may seem  a little restrictive, however, there   are many  known conditions for a norm being equivalent to a norm generated by an inner product. See for instance \cite[Theorem 7.12.68]{istratescu} and \cite[Theorem 3.3]{jacek}.

Proposition \ref{derecho} and Corollary \ref{corol} yield the following corollary.

\begin{corollary}\label{linh}
Let $(X,q)$ be a right-bounded asymmetric normed space such that $X$ satisfies (1) and (2) from Proposition \ref{derecho}. If $X/\left<{\theta_q}\right>$ is finite-dimensional, then $X$ is linearly homeomorphic to $\left<{\theta_q}\right>\times Z$, where $Z$ is a $T_1$ subspace of $(X,q)$, linearly homeomorphic to $X/\left<{\theta_q}\right>$. In particular $Z$ is linearly homeomorphic to a Euclidean space.

\end{corollary}
Since every finite-dimensional asymmetric normed space satisfies (1) and (2) from Proposition \ref{derecho}, we directly conclude the following.

\begin{corollary}\label{linh 2}
If $(X,q)$ is a right-bounded finite-dimensional asymmetric normed space, then $X$ is linearly homeomorphic to $\left<{\theta_q}\right>\times Z$, where $Z$ is a Euclidean subspace of $(X,q) $  linearly homeomorphic to $X/\left<{\theta_q}\right>$.
\end{corollary}

The following example shows that the condition of being right-bounded is essential in Corollaries  \ref{linh} and \ref{linh 2}, even in the finite-dimensional case.

\begin{example}
In $\mathbb{R}^2$, consider the asymmetric norms $q$ and $p$ defined as in Example \ref{parabola}. We have $\theta_p=\theta_q$, and since $p^s(x,y)=\max\{|x|,|y|\}$, it is easy to see that $$B_p(0,1)=B_{p^s}(0,1)+\theta_p,$$ so $(\mathbb{R}^2,p)$ is right-bounded. Corollary \ref{linh} implies that $(\mathbb{R}^2,p)$ is linearly homeomorphic to $\left<{\theta_p}\right>\times\mathbb{R}$, where $\mathbb{R}$ has its usual topology. 

On the other hand, $q$ is not right-bounded (See \cite[Example 2.7 and Proposition 2.9]{Conradie 2}) and $\left<{\theta_q}\right>\times\mathbb{R}=\left<{\theta_p}\right>\times\mathbb{R}$ is homeomorphic to $(\mathbb{R}^2,p)$, which is not homeomorphic to $(\mathbb{R}^2,q)$. Indeed, all open sets in $(\mathbb{R}^2,q)$ have $q$-closure equal to the whole space $\mathbb{R}^2$, while the  $p$-closure of the ball $B_p(0,1)$ is the set $\{(x,y):-1\le x\le 1\}\neq\mathbb{R}^2$.
\end{example}

\section{Covering dimension vs. algebraic dimension}

In this section, we will find the covering dimension of some asymmetric normed spaces. We refer the reader to \cite{pears} for more information about the theory of dimension. We emphasize that the definition of covering dimension that we will use (c.f. \cite{pears})  does not require any additional condition on the topological space, such as being Hausdorff or any other separation axiom.

The {\it order} of a family of nonempty subsets $\{A_i:i\in I\}$  is the maximum $n\in\mathbb{N}\cup\{0\}$ for which there exists $M\subseteq I$ with $|M|=n+1$ such that $\bigcap_{i\in M}A_i\neq\emptyset$. If there is no such $n$, then we say that the order of $\{A_i:i\in I\}$  is  $\infty$. The {\it covering dimension} of a topological space $X$ is the minimum $n\in\mathbb{N}\cup\{0\}$ such that every finite open cover of $X$ has an open refinement of order not exceeding $n$, or is $\infty$ if there is no such $n$. We will denote this number with $\dim X$. On the other hand, if $X$ is a linear space then $\dim_aX$ will denote the algebraic dimension, namely the cardinality of any basis of $X$. As we have done before, when we say that $X$ is finite-dimensional we understand that $\dim_aX<\infty$.
 
It is well known that for every finite-dimensional normed space $X$ the covering dimension coincides with its algebraic dimension. Furthermore, if $X$ is an infinite-dimensional normed space, then $\dim X=\infty$.
Thus, it is natural to ask what happens when we consider asymmetric normed spaces.
As expected, in this case the situation is rather different.  

For example, if we consider on  $\mathbb{R}^m$ the asymmetric norm $q:\mathbb{R}^m\to\mathbb{R}^+$ defined by $q(x_1,\ldots,x_m)=\max\{x_1^+,\ldots,x_m^+\}$, then it is not difficult to show that $(\mathbb{R}^m,q)$ has covering dimension zero. In fact, every finite open cover must have the whole space $\mathbb{R}^m$ between its elements. The essential reason for this is that $\theta_q$ is somehow big, meaning (in this case) that it contains nonempty $q$-open sets. We generalize this situation in the following lemmas.

\begin{lemma}\label{nonemp}
Let $(X,q)$ be an asymmetric normed space such that there exist $x\in X$ and $\epsilon>0$ with $B_q(x,\epsilon)\subseteq\theta_q$. If $\mathcal{U}$ is any finite $q$-open cover, then $X\in\mathcal{U}$. 
\end{lemma}

\begin{proof}
Since $-nx\in X$ for every $n\in\mathbb{N}$ there must exist an infinite sequence of natural numbers $(n_j)_{j\in\mathbb{N}}$ and $U\in\mathcal{U}$ such that $-n_jx\in U$ for every $j\in\mathbb{N}$. Let $r_j>0$ be such that $B_q(-n_jx,r_j)\subseteq U$ for every $j\in\mathbb{N}$. Fix any $y\in X$, and let $m\in\mathbb{N}$ with $m>2q(y)$. It is easy to see that $$\frac{\epsilon y}{m}+x\in B_q(x,\epsilon).$$

Since $B_q(x,\epsilon)\subseteq\theta_q$ it follows that $q(\epsilon y+mx)=0$. This also implies $$0=q(\epsilon y+mx)=q(\epsilon y+(m+1)x-x),$$ so $\epsilon y+(m+1)x\in B_q(x,\epsilon)\subseteq\theta_q$ and then $q(\epsilon y+(m+1)x)=0$. By induction we can see that $q(\epsilon y+(m+k)x)=0$ for each $k\in\mathbb{N}$. Pick $j\in\mathbb{N}$ with $n_j>m$, so in particular for $k=n_j-m$ we obtain $q(\epsilon y+n_jx)=0$. This means $$\epsilon y\in B_q(-n_jx,r_j)\subseteq U.$$

Since $y$ was arbitrary, we obtain $X=U$.
\end{proof}

We will show now that Lemma \ref{nonemp} still holds if we only suppose that $B_{q^s}(x,\epsilon)\subseteq\theta_q$ (a weaker condition, because we always have $B_{q^s}(x,\epsilon)\subseteq B_{q}(x,\epsilon)$). 

\begin{lemma}\label{nonempdos}
Let $(X,q)$ be an asymmetric normed space such that there exist $x\in X$ and $\epsilon>0$ with $B_{q^s}(x,\epsilon)\subseteq\theta_q$. If $\mathcal{U}$ is any finite $q$-open cover, then $X\in\mathcal{U}$. 
\end{lemma}

\begin{proof}
In \cite[Lemma 2.8]{Conradie 2}, it was shown that there is an asymmetric norm $p$ on $X$ (not necessarily equivalent to $q$) such that $$B_p(0,1)=B_{q^s}(0,1)+\theta_q.$$

Moreover, $\theta_p=\theta_q$ (see \cite[Lemma 2.6]{Jonard Sanchez 3}) and the topology induced by $p$ is finer than that induced by $q$ (\cite[Proposition 2.6]{Conradie 2}).
Hence, $\mathcal{U}$ is a finite $p$-open cover of $X$. Since $$B_p(x,\epsilon)=B_{q^s}(x,\epsilon)+\theta_q\subseteq \theta_q+\theta_q=\theta_p,$$then by Proposition \ref{nonemp} we conclude that $X\in\mathcal U$.
\end{proof}

\begin{corollary}\label{zerodim}
Let $(X,q)$ be an asymmetric normed space such that $\theta_q$ has nonempty $q^s$-interior. Then $\dim X=0$.
\end{corollary}

\begin{proof}
If $\mathcal{U}$ is a finite $q$-open cover of $X$, then $X\in\mathcal{U}$ by Lemma \ref{nonempdos}. Therefore $\{X\}$ is an open refinement of $\mathcal{U}$ whose order is $0$.
\end{proof}

The following lemma is a direct consequence of \cite[Lemma 3.2]{Aliprantis}.

\begin{lemma}\label{geo}
Let $(X,q)$ be a finite-dimensional asymmetric normed space. Then $\theta_q$ spans the whole space $X$ if and only if $\theta_q$ has nonempty $q^s$-interior.
\end{lemma}

The following characterization of the covering dimension is very useful \cite[Chapter 3, Proposition 1.2]{pears}.

\begin{proposition}\label{pears}
For any topological space $X$ the following statements are equivalent:
\begin{enumerate}[\rm(1)]
\item $\dim X\le n,$
\item If $\{U_1,\ldots,U_{n+2}\}$ is an open cover of $X$, there is an open cover $\{V_1,\ldots,V_{n+2}\}$ such that each $V_i\subseteq U_i$ and $\bigcap_{i=1}^{n+2}V_i=\emptyset$.
\end{enumerate}
\end{proposition}

\begin{lemma}\label{cone}
Let $(X,q)$ be a finite-dimensional asymmetric normed space, with $0\neq Y=\left<{\theta_q}\right>\neq X$. Let $\left\|{\cdot}\right\|$ be a norm on $X$ induced by a fixed inner product, and suppose that $X=Y\oplus{Z}$, where $Z=Y^{\perp}$ is the orthogonal complement of $Y$ with respect to the inner product. Let $C$ and $\{z_j:j\in\mathbb{N}\}$ be subsets of $Z$ such that $z_j\notin C$ for every $j\in\mathbb{N}$. Then there exists an open subset $U$ of $X$ with the following two properties:
\begin{enumerate}[\rm(1)]
\item $Y+C\subseteq U,$
\item $Y+z_j$ is not contained in $U$ for every $j\in\mathbb{N}$.
\end{enumerate}
\end{lemma}

\begin{proof}
First, observe that by Lemma \ref{zeta}, the subspace $Z$ is $T_1$ and normable by $q^s$. Hence $q(z_j-c)>0$ for every $j\in\mathbb{N}$ and $c\in C$, because $z_j-c\in Z\setminus\{0\}$.

Since $X$ is finite-dimensional, we can find $N>0$ such that 
$$N\left\|{x}\right\|\le q^s(x)$$
for every $x\in X$. Following the same arguments used in   the proof of Proposition~\ref{derecho}, we get that 
\begin{equation}\label{cinco}
N\left\|{y}\right\|\le q^s(y+z).
\end{equation}
 every $y\in Y$ and $z\in Z$ (c.f. inequality (\ref{d:norma y qs}) of Proposition~\ref{derecho}).
Since $(Z,q^{-})$ is also normable by $q^s$, then $q$ and $q^-$ are equivalent, so there exists $M>0$ such that
\begin{equation}\label{uno}
Mq(-z)\le q(z),\; \text{ for every  }z\in Z.
\end{equation}

Let $\{y_1,\ldots,y_m\}\subseteq\theta_q$ be a linear basis of $Y$. Define $y=\sum_{i=1}^my_i\in\theta_q$. By Lemma \ref{suma} $y\neq 0$. For every $z\in Z$, $n\in\mathbb{N}$ and $r>0$, consider $$I_{n,z,r}=B_q(-ny+z,r)\cap Z.$$
The sets $I_{n,z,r}$ are open in the Euclidean space $Z$. Clearly, if $r<r'$ then $I_{n,z,r}\subseteq I_{n,z,r'}$. Note that $$\bigcap_{r>0}B_q(-ny+z,r)=\theta_q+(-ny+z),$$ so it follows that $$\bigcap_{r>0}I_{n,z,r}=\{z\}.$$

Denote the diameter of each $I_{n,z,r}$ with respect to the norm $q^s$ by $D(I_{n,z,r})$; namely $D(I_{n,z,r})=\sup\{q^s(z_1-z_2):z_1,z_2\in I_{n,z,r}\}$. Therefore

$$\lim_{r\to 0^+}D(I_{n,z,r})=0.$$

We will use this fact to construct the desired set $U$ as follows. For each $n\in\mathbb{N}$ and $c\in C$, choose $r_{n,c}>0$ such that
\begin{enumerate}[(a)]
\item $\max\{r_{n,c},r_{n,c}/M\}<N\left\|{y}\right\|$. %\label{cuatro}
\item $D(I_{n,c,r_{n,c}})<\min\{q(z_j-c):j\le n\}.$ %\label{seis}
\end{enumerate}

Let $B_{n,c}=B_q(-ny+c,r_{n,c})$.

{\it Claim I:} $Y+c\subseteq\bigcup_{n\in\mathbb{N}}B_{n,c}$.

Indeed, an element in $Y+c$ has the form $y_0+c$, where $y_0=\sum_{i=1}^mt_iy_i$ for some $t_1,\ldots,t_m\in\mathbb{R}$. Let $n\in\mathbb{N}$ be such that $t_i+n>0$ for all $i=1,\ldots,m$. Then $$q(y_0+c-(-ny+c))=q(y_0+ny)\le \sum_{i=1}^m(t_i+n)q(y_i)=0<r_{n,c}.$$
Hence $y_0+c\in B_{n,c}$. This proves $Y+c\subseteq\bigcup_{n\in\mathbb{N}}B_{n,c}$, as desired.

Let $$U=\bigcup_{c\in C}\bigcup_{n\in\mathbb{N}}B_{n,c}.$$ It follows from Claim I that $Y+C\subseteq U$. 

Fix $j\in\mathbb{N}$. In order to prove that $Y+z_j$ is not contained in $U$, it suffices to show that $-jy+z_j\notin{U}$. 

Suppose this is not the case, so let us assume that there exist $n\in\mathbb{N}$ and $c\in C$ such that $-jy+z_j\in B_{n,c}$. This means that
\begin{equation}
q(-jy+z_j-(-ny+c))=q((-j+n)y+z_j-c)<r_{n,c}.
\end{equation}
Let us call $w=(-j+n)y+z_j-c$, so $q(w)<r_{n,c}$.

{\it Claim II:} $j\le n$.

Assume that the opposite is true, so $j-n>0$. In particular, $(j-n)y\in\theta_q$. By using the inequality given in Lemma \ref{suma}, we obtain
\begin{equation}\label{dos}
q(w)=q((-j+n)y+z_j-c)\ge q(z_j-c)-q((j-n)y)=q(z_j-c).
\end{equation}

On the other hand,
\begin{equation}\label{tres}
q(-w)=q((j-n)y-z_j+c)\le q((j-n)y)+q(-z_j+c)=q(-z_j+c).
\end{equation}

By (\ref{uno}), we have $Mq(-z_j+c)\le q(z_j-c)$. Hence, by (\ref{dos}) and (\ref{tres}) we conclude that $Mq(-w)\le q(w)<r_{n,c}$. Therefore, by (\ref{cinco}) and the definition of $r_{n,c}$, 
$$N|-j+n|\left\|{y}\right\|=N\left\|{(-j+n)y}\right\|\le q^s(w)<\max\{r_{n,c},r_{n,c}/M\}<N\left\|{y}\right\|.$$
This implies that $|-j+n|<1$, a contradiction since $j\neq n$. Thus the Claim II is proved.

Note that $z_j\in B_{n,c}$, because $$q(z_j-(-ny+c))=q(z_j+ny-c-jy+jy)\le q(w)+q(jy)=q(w)<r_{n,c}.$$

Analogously, $c\in B_{n,c}$, because $$q(c-(-ny+c))=q(ny)=0<r_{n,c}.$$

We conclude that $c,z_j\in I_{n,c,r_{n,c}}$. But by Claim II $j\le n$, so part (b) of  the definition of $r_{n,c}$ implies that $$q(z_j-c)\le q^s(z_j-c)\le D(I_{n,c,r_{n,c}})<q(z_j-c).$$

This contradiction shows that $-jn+z_j$ cannot be in $B_{n,z}$, and now the proof is complete.
\end{proof}

\begin{lemma}\label{reban}
Let $(X,q)$ be an asymmetric normed space and $Y=\left<{\theta_q}\right>$. Assume that $Z\subset X$ is a linear subspace such that $X=Y+Z$. If $Y$ is finite-dimensional then for every  finite open cover  $\mathcal{U}$ of $X$ and  every $z\in Z$ there exists $U\in\mathcal{U}$ such that $Y+z\subseteq U$.
\end{lemma}

\begin{proof}
Take a basis $\{y_1,\ldots,y_k\}\subseteq\theta_q$ of $Y$. For every $n\in\mathbb{N}$ consider $$x_n=-\sum_{i=1}^kny_i+z.$$

Since the cover $\mathcal{U}$ is finite, there exists $U\in\mathcal{U}$ containing infinitely many elements of $\{x_n:n\in\mathbb{N}\}$. Let $(n_j)_{j\in\mathbb{N}}$ be an increasing sequence of natural numbers such that $x_{n_j}\in U$ for every $j\in\mathbb{N}$. If $j\in\mathbb{N}$, since $U$ is open there exists $\epsilon_j>0$ such that $$B_q\bigg(-\sum_{i=1}^kn_jy_i+z,\epsilon_j\bigg)\subseteq U.$$ 

Let $y+z\in X$. So $y=\sum_{i=1}^kt_iy_i$ for some $t_1,\ldots,t_k\in\mathbb{R}$. Let $j\in\mathbb{N}$ such that $t_i+n_j>0$ for every $i\in\{1,\ldots,k\}$. Then 
$$q\bigg(\sum_{i=1}^kt_iy_i+z+\sum_{i=1}^kn_jy_i-z\bigg)\le\sum_{i=1}^k(t_i+n_j)q(y_i)=0.$$

So $$y+z=\sum_{i=1}^kt_iy_i+z\in B_q\bigg(-\sum_{i=1}^kn_jy_i+z,\epsilon_j\bigg)\subseteq U.$$

Hence $Y+z\subseteq U$.
\end{proof}

\begin{theorem}\label{infin}
Let $(X,q)$ be a finite-dimensional asymmetric normed  space such that $0\neq Y=\left<{\theta_q}\right>\neq X$. Then $\dim X=\infty$.  
\end{theorem}

\begin{proof}
Fix $n\in\mathbb{N}$. We will use Proposition \ref{pears} to show that $\dim X\le n$ is false. Consider an inner product on $X$ so that $X=Y\oplus{Z}$, where $Z=Y^{\perp}$. By Lemma \ref{zeta} $(Z,q)$ is normable, so it is homeomorphic to some Euclidean space. Let $Z_1,\ldots,Z_{n+2}\subseteq Z$ be pairwise disjoint dense subsets of $Z$ with $Z=\bigcup_{k=1}^{n+2}Z_i$. For each $k\in\{1,\ldots,{n+2}\}$, let $\{z_j^{(k)}:j\in\mathbb{N}\}$ be a countable dense set of $Z_k$ which is also dense in $Z$. By Lemma \ref{cone}, for every $k\in\{1,\ldots,{n+2}\}$ there is an open subset $U_k$ of $X$ such that $U_k$ contains $Y+Z_k$ and $U$ does not contain $Y+z_{j}^{(i)}$ for every $j\in\mathbb{N}$ and $i\in\{1,\ldots,{n+2}\}\setminus\{k\}$. Clearly $$X=Y+Z=Y+\bigcup_{k=1}^{n+2}Z_k=\bigcup_{k=1}^{n+2}U_k.$$

Hence $\{U_1,\ldots,U_{n+2}\}$ is an open cover of $X$. In order to apply Proposition \ref{pears}, let $\{V_1,\ldots,V_{n+2}\}$ be any open cover of $X$ such that each $V_k\subseteq U_k$ and let us show that $\bigcap_{k=1}^{n+2}V_k\neq\emptyset$. Consider the element $z_1^{(1)}$. By Lemma \ref{reban} some $V_k$ must contain $Y+z_1^{(1)}$, and so does $U_k$. Since only $U_1$ contains $Y+z_1^{(1)}$ then $k=1$. In particular $z_1^{(1)}\in V_1$. Let $m<n+2$ and suppose there is $j$ such that $z_j^{(m)}\in\bigcap_{k=1}^mV_k$. Since  $\bigcap_{k=1}^mV_k$ is $q$-open, there exists $\epsilon_m>0$ such that $$B_q(z_j^{(m)},\epsilon_m)\subseteq\bigcap_{k=1}^mV_k.$$

Since $\{z_i^{(m+1)}\}$ is dense, there is some $z_{j_0}^{(m+1)}\in B_q(z_j^{(m)},\epsilon_m)$, so $z_{j_0}^{(m+1)}\in\bigcap_{k=1}^mV_k$. Again, $Y+z_{j_0}^{(m+1)}$ is contained in some element of $\{V_1,\ldots,V_{n+2}\}$, it only can be $V_{m+1}$ by construction. In particular $z_{j_0}^{(m+1)}\in V_{m+1}$, so $$\bigcap_{k=1}^{m+1}V_k\neq\emptyset.$$

This proves $\bigcap_{k=1}^{n+2}V_k\neq\emptyset.$
\end{proof}

We have thus computed the covering dimension of all finite-dimensional asymmetric normed spaces.

\begin{theorem}\label{t:main dimension}
Let $(X,q)$ be a finite-dimensional asymmetric normed space. Let $Y=\left<{\theta_q}\right>$.
\begin{enumerate}[\rm(1)]
\item If $Y=X$, then $\dim X=0$.
\item If $0\neq Y\neq X$, then $\dim X=\infty$.
\item If $Y=0$, then $\dim X = \dim _aX$.
\end{enumerate}
\end{theorem}

\begin{proof}
If $Y=X$, by Lemma \ref{geo}, $\theta_q$ has nonempty $q^s$-interior. Hence, by Corollary \ref{zerodim}, $\dim X=0$.

If $0\neq Y\neq X$, then $\dim X=\infty$ by Theorem \ref{infin}.

If $Y=0$, then $\theta_q=\{0\}$. Hence $X$ is $T_1$, so it is normable and homeomorphic to the Euclidean space $\mathbb{R}^n$, where $n=\dim_a X$. Since  $\mathbb{R}^n$ has covering dimension $n$ ( \cite[Chapter 3, Theorem 2.7]{pears}) so does $X$.
\end{proof}

\section{Separation axioms on asymmetric normed spaces}

By Proposition \ref{norm}, every  finite-dimensional asymmetric normed space satisfying the axiom $T_1$ is normable, which in particular means that, in this case, $T_1$ implies $T_2$. In this section we explore the relation between other separation axioms that occur in asymmetric normed spaces.

Let us begin by remembering that, if $(X,q)$ is an asymmetric normed space, then $X$ is always a $T_0$ space.
% because $q(x)=0=q(-x)$ if and only if $x=0$. Also, by Proposition \ref{haus}, if $q$ satisfies $q(x)>0$ for all $x\neq 0$ (equivalently, $\theta_q=\{0\}$), then $X$ is a $T_1$ space.

In \cite{arenas}, F. G. Arenas, J. Dontchev and M. Ganster introduced the $T_{\frac{1}{4}}$ separation axiom as follows.

\begin{definition}
Let $X$ be a topological space. We say that $X$ is a $T_{\frac{1}{4}}$ space if for any $x\in X$ and any finite collection of points $x_1,\ldots,x_n\in X$, with $x\neq x_i$ for all $i=1,\ldots,n$, there exists an open set $U\subseteq X$ such that $x\in U$ and $x_1,\ldots,x_n\notin U$, or $x\notin U$ and $x_1,\ldots,x_n\in U$.
\end{definition}

It is easy to see that $T_1$ implies $T_{\frac{1}{4}}$, $T_{\frac{1}{4}}$ implies $T_0$, and that these axioms are non-equivalent. As the notation suggests, the $T_{\frac{1}{4}}$ axiom can be thought as the weakest separation axiom that implies $T_0$. It is natural to ask if there exists a $T_{\frac{1}{4}}$ asymmetric normed space that is not $T_1$. We give a negative answer to this.

\begin{proposition}
Every $T_{\frac{1}{4}}$ asymmetric normed space $(X,q)$ is $T_1$.
\end{proposition}

\begin{proof}
Let us suppose that $(X,q)$ is not $T_1$ and let us show that $X$ cannot be $T_{\frac{1}{4}}$. Then, pick $x\in\theta_q$ with $x\neq 0$. Clearly, any basic neighborhood of $0$, say $B_q(0,\epsilon)$, contains $x$ because $q(x)=0<\epsilon$. On the other hand, any neighborhood of $x$, say $B_q(x,\delta)$, contains $2x$ because $q(2x-x)=0$. Hence, there is no $q$-open set $U$ such that $2x,0\in U$ and $x\notin U$, or $2x,0\notin U$ and $x\in U$.
\end{proof}

T. Banakh and A. Ravsky proved in  \cite{ravs} that 
every $T_3$ paratopological group is $T_{3\frac{1}{2}}$.
This result in combination
with Propositions \ref{prop bola cerrada regular} and Corollary \ref{corol t3 implica t2}, yields the following.

\begin{corollary}
Let $(X,q)$ be an asymmetric normed space. If the closed ball $B_q[0,1]$ is a $q$-closed set in $X$, then $X$ is $T_{3\frac{1}{2}}$ and $T_2$. Hence, $X$ is completely regular. 
\end{corollary}

\begin{proposition}\label{p:regular bounded}
If $(X,q)$ is a regular space, then  $\overline{B_q[0,1]}$ is $q$-bounded. Namely, there exists $k>0$ such that $\overline{B_q[0,1]}\subset B_q(0, k)$.
\end{proposition}

\begin{proof}
If $X$ is regular, there exists $r>0$ such that $\overline{B_q(0,r)}\subseteq B_q(0,1)$. Therefore $\overline{B_q[0,r/2]}\subseteq B_q(0,1)$. Since the function $x\mapsto (2/r)x$ is a homeomorphism on $(X,q)$, we conclude that $\overline{B_q[0,1]}\subseteq B_q(0,2/r)$.
\end{proof}

\begin{example}\label{cec}
Consider the linear space $X=C_0[0,1]$ of all continuous functions $f:[0,1]\to\mathbb{R}$ such that $\int_{0}^{1}f(t)dt=0$, with the asymmetric norm $q(f)=\max\{f(t):t\in [0,1]\}$. Then $B_q[0,1]$ is $q$-closed in $X$, and therefore $X$ is completely regular. In particular it is Hausdorff.
\end{example}

\begin{proof}
Let $B=B_q[0,1]$. Suppose that $(f_n)\subseteq B$ is a sequence that converges to $f$, for some $f\in X$. We will show that $q(f)\le 1$. Suppose on the contrary that $q(f)>1$, then there exists $t_0\in [0,1]$ such that $f(t_0)>1$. For all $n\in\mathbb{N}$, we have $$0=\int_{0}^{1}(f_n(t)-f(t))dt=\int_{0}^{1}(f_n(t)-f(t))^+dt-\int_{0}^{1}(f_n(t)-f(t))^-dt,$$ where the positive and negative parts of a real-valued function $g$ are defined as $g^+=\max\{g,0\}$ and $g^-=\max\{-g,0\}.$
Since $q(f_n-f)\to 0$, then $\int_{0}^{1}(f_n(t)-f(t))^+dt$ converges to zero. Hence $\int_{0}^{1}(f_n(t)-f(t))^-dt$ must converge to zero too. Since $$f(t_0)>1\ge q(f_n)\ge f_n(t)$$ for all $n$ and $t\in [0,1]$, there is a closed subinterval $J\subseteq [0,1]$ containing the point $t_0$ in its interior, where $f(t)-f_n(t)>\frac{1}{2}(f(t_0)-1)>0$ for all $n$ and $t\in J$. Thus $$\int_{0}^{1}(f_n(t)-f(t))^-dt=\int_{0}^{1}\max\{f(t)-f_n(t),0\}dt\ge \int _J \frac{1}{2}(f(t_0)-1)dt>0.$$
This contradicts the fact that $\int_{0}^{1}(f_n(t)-f(t))^+dt$ converges to zero.
\end{proof}

By Lemma~\ref{acota}, if $B_q[0,1]$ is $q^s$-bounded then $(X,q)$  is normable, hence it is $T_2$.
It is well known that the asymmetric normed space described in Example~\ref{cec} is not a topological vector space and hence it is not normable (see, e.g., \cite[Example 1.1.41]{cobzas}). Therefore, its  closed unit ball cannot be $q^s$-bounded.   This shows that in the infinite-dimensional case, not every Hausdorff (or even, completely regular) asymmetric normed space is normable.

We will show now that if the set $\theta_q$ has nonempty $q^s$-interior then the space is $T_4$. Although it seems that, under this condition there cannot be disjoint nonempty closed subsets, it is worth observing that this can be deduced from our results about the covering dimension of asymmetric normed spaces. Indeed, it is well known that any topological space $X$ such that $\dim X=0$ is a $T_4$-space (see \cite[Chapter 3]{pears}). Then, by  Corollary \ref{zerodim} we can infer the following.

\begin{corollary}\label{tcuat}
Let $(X,q)$ be an asymmetric normed space such that $\theta_q$ has nonempty $q^s$-interior. Then $(X,q)$ is a $T_4$-space.
\end{corollary}

In the following example, we show an asymmetric normed space such that it is not a $T_3$-space nor a $T_4$-space.

\begin{example}
Let $X=\mathbb{R}^2$ and consider the asymmetric norm $q(x,y)=\max\{x^+,|y|\}$. Then $X$ is not a $T_3$-space nor a $T_4$-space.
\end{example}
\begin{proof}
Let $F$ be the closure of the set $\{(x,1/x):x>0\}$ and 
$C=\{(x,0):x\in\mathbb{R}\}$. It is easy to see that $F\cap C=\emptyset$ and that $C$ is a $q$-closed subset.  Consider an open neighborhood $B_q((0,0),\epsilon)$ of $(0,0)$ and $A$ an open set containing $F$. Clearly $(0,\epsilon/2)\in B_q((0,0),\epsilon)$. Also $(2/\epsilon , \epsilon/2)\in F$, so there is $\delta >0$ such that $$B_q((2/\epsilon,\epsilon/2),\delta)\subseteq A.$$ 

Since $$q((0,\epsilon/2)-(2/\epsilon ,\epsilon/2))=0<\delta$$ we have $(0,\epsilon /2)\in A$. Hence any two open sets containing $(0,0)$ and $F$, respectively, cannot be disjoint. This shows that $X$ is not a $T_3$-space. The same reasoning shows there cannot be two open disjoint sets containing $F$ and $C$, respectively.
\end{proof}

In the literature about asymmetric normed spaces, every $T_1$ infinite-dimensional asymmetric normed space that we find has closed unit ball, and therefore it is completely regular (and Hausdorff). From this, we find interesting the following examples.

\begin{example}\label{e: taras t1}
There is a $T_1$ infinite-dimensional asymmetric normed space which is not Hausdorff.
\end{example}

\begin{proof}

Let $\{e_n\}_{n=1}^{\infty}$ be the standard orthonormal Schauder basis of $\ell_2$. Consider the set $$S=\{-e_n,ne_n,(n+1)(e_{n+1}+e_1)\}_{n=1}^{\infty}.$$

Let $B$ be the convex hull of $S$, and let $X$ be the linear span of $B$. Define  the asymmetric norm $q$ on $X$ as the Minkowski gauge functional on $B$. 

First we will prove that $(X,q)$ is not Hausdorff. For every $n\in\mathbb{N}$, the  point $(n+1)(e_{n+1}+e_1)\in S$  and then $q((n+1)(e_{n+1}+e_1))\le 1$ . Hence $$q(e_{n+1}+e_1)\le 1/(n+1), $$  and therefore the sequence $(e_{n+1}+e_1)$ converges to $0$ in $(X,q)$. On the other hand, since $(n+1)e_{n+1}\in S$ for all $n\in\mathbb{N}$, then 
$$q\big((e_{n+1}+e_1)-e_1\big)=q(e_{n-1})=\frac{q((n+1)e_{n+1})}{n+1} \longrightarrow 0.$$ 
From this we infer that the sequence $(e_{n+1}+e_1)$ converges to both $0$ and $e_1$, and then $(X,q)$ is not a Hausdorff space.

To prove that $(X,q)$ is $T_1$ it suffices to show that  $B$ does not contain any ray $\{tx:t\ge 0\}$ with $x\in B\setminus\{0\}$.

Observe that every point $x=(x_n)\in B$,  can be written  as the convex sum
\begin{equation}\label{e:descripcion x}
x=-\sum_{n=1}^{\infty} s_ne_{n}+\sum_{n=1}^{\infty} t_nne_{n}+\sum_{n=1}^{\infty} r_n(n+1)(e_{n+1}+e_1),
\end{equation}
where $\{s_n, t_n, r_n\}_{n\in\mathbb N}\subset [0,1]$, $\sum\limits_{n=1}^\infty (s_n+t_n+r_n)=1$, and all but finitely many $s_n,t_n,r_n$ are zero. Hence, the coordinates of the vector $x$ can be written as
$$x_1=-s_1+t_1+\sum_{n=1}^{\infty}(n+1)r_n,$$
$$x_k=-s_k+kt_k+kr_{k-1}, \quad \text{if } k\ge 2.$$

From this representation we conclude that if $x=(x_n)\in B$, then
$-1\leq x_n$, for every $n\in\mathbb N$. Furthermore, if $k\geq 2$ then $x_k\leq k$. 
Let us assume that  $x\in B\setminus \{0\}$ is such that $\{tx:t\ge 0\}\subset B$. Then   $x_1>0$ and $x_k=0$ for every $k\geq 2$. Since the whole ray $\{tx:t\ge 0\}$ is contained in $B$, we can assume without loss of generality that $x_1>1$. 
Thus, if $n\geq 2$, we have that 
$$0=x_n=-s_n+nt_n+nr_{n-1}\geq-s_n+ nr_{n-1}.$$
Then $s_n\geq  nr_{n-1}$ and therefore:

$$1<x_1 = -s_1+t_1+\sum_{n=1}^{\infty}(n+1)r_n \le -s_1+t_1+\sum_{n=2}^{\infty}s_n \le 1.$$
This contradiction concludes the proof.

\end{proof}

\begin{example}\label{e: taras regular}
There exists a Hausdorff asymmetric normed space $(X,q)$ such that its closed unit ball $B_q[0,1]$ is not $q$-closed and  $X$ is not regular. 
\end{example}

\begin{proof}
Let $(X,\left\|{\cdot}\right\|)$ be a normed space with a discontinuous linear functional $f:X\to\mathbb{R}$.  The map $q:X\to \mathbb R$ defined by $q(x)=\max\{\left\|{x}\right\|,f(x)\}$ is an asymmetric norm on $X$. Consider $B_1 = B_q[0,1]$ and $B_2=B_{\left\|{\cdot}\right\|}[0,1]$. We will show that $\overline{B_1}=B_2$, where $\overline{B_1}$ is the $q$-closure of $B_1$.

First, note that  $\left\|{x}\right\|\le q(x)$ for all $x\in X$.  Then $B_1\subseteq B_2$ and the topology of the norm $\left\|{\cdot}\right\|$ is contained in the topology of the asymmetric norm $q$. This implies that $(X,q)$ is Hausdorff and $B_2$ is closed in $(X,q)$.  So, $\overline{B_1}\subseteq B_2$.

In order to prove the other contention, consider any point $x\in B_2$ and assume without loss of generality  that $f(x)\ge 0$ (the other case is analogous).
 Since $f$ is not continuous, there exists a sequence $(x_n)\subseteq B_2$ such that $(f(x_n))$ diverges to $+\infty$. 
 For every $n\in\mathbb{N}$, let $y_n$ be defined by $$y_n=-\frac{x_n}{f(x_n)}f(x)+x.$$
It is clear that $\|y_n- x\|\to 0$. 
On the other hand $f(y_n-x)=-f(x)\le 0$ for all $n\in\mathbb{N}$, so $q(y_n-x)=\left\|{y_n-x}\right\|\to 0$.

If $\left\|{x}\right\|<1$ we can suppose that $(y_n)\subseteq B_2$. Also, $f(y_n)=-f(x)+f(x)=0$. This shows that $q(y_n)=\|y_n\|\le 1$, so $(y_n)\subseteq B_1$ and therefore $x\in\overline{B_1}$. 

If $\left\|{x}\right\|=1$. Let $z_n=(1-1/n)x$ for every $n\in\mathbb{N}$. Hence $\left\|{z_n}\right\|=1-1/n<1$, so this sequence is contained in $\overline{B_1}$, by the previous case. Since $q(z_n-x)=(1/n)q(-x)\to 0$, then $x\in \overline{\overline{B_1}}=\overline{B_1}$, as desired. Thus $B_2=\overline{B_1}$.

Finally observe that $f(B_2)$ is not bounded in $\mathbb{R}$, and therefore the set $B_2=\overline{B_1}$ is not $q$-bounded. This, in combination with Proposition~\ref{p:regular bounded} implies that $(X,q)$ is not regular and therefore $B_1$ is not $q$-closed neither. 
\end{proof}

\subsection{Final Question}

For asymmetric normed spaces, the relations between axioms $T_1$, $T_2$, $T_3$, $T_{3\frac{1}{2}}$ and the property of having the closed unit ball $B_q[0,1]$ $q$-closed are summarized in the following diagram

{ 
\[ \begin{tikzcd}
B_q[0,1] \text{ is closed} \arrow[Rightarrow]{r} & T_3 \arrow[Leftrightarrow]{r} \arrow[Rightarrow]{d} & T_{3\frac{1}{2}} \\
 & T_2 \arrow[Rightarrow]{d} \\
 & T_1 
\end{tikzcd}
\]
}

As we have shown in Examples~\ref{e: taras t1} and \ref{e: taras regular}, implications $T_2\Rightarrow T_1$ and $T_3\Rightarrow T_2$ cannot be reversed. 

On the other hand, the property of having the closed unit ball $B_q[0,1]$ $q$-closed is equivalent to the continuity of the norm $q:(X, q)\to\mathbb (\mathbb R, |\cdot|)$. Indeed, for every $a\in\mathbb R$, the set $q^{-1}\left((-\infty, a)\right)=B_q(0,a)$ is always open and therefore $q$ is always upper semi-continuous. On the other hand $B_q[0,1]$ is $q$-closed if and only if  $B_q[0,a]$ is closed for every $a\in\mathbb R$. This happens if and only if the complement of  $B_q[0,a]$ is $q$-open, namely, iff $q^{-1}\left((a,\infty)\right)$ is a $q$-open set. Thus, closedness of  $B_q[0,1]$ is equivalent to the lower semi-continuity of $q$, which in turn is equivalent to the continuity of $q$ (c.f. \cite[Proposition 1.1.8 (5)]{cobzas}).

From this situation we are particularly interested in the following natural question:

\begin{question} 
Does every regular asymmetric normed space admit an equivalent asymmetric norm which is continuous?
\end{question}

\begin{center}
ACKNOWLEDGMENTS
\end{center}
\textit{The authors  want to express their gratitude to Taras Banakh, for suggesting  Examples~\ref{e: taras t1} and \ref{e: taras regular}.
The authors also want to thank the anonymous referees for the careful reading  and the suggestions that helped  improve the final version of this paper.  }

\vspace{2cm}

%\noindent[Enrique A. S\'anchez P\'erez] Instituto Universitario de Matem\'{a}tica Pura y Aplicada, Universitat Polit\`ecnica de Val\`encia, Camino de Vera s/n, 46022 Valencia, Spain, e-mail: easancpe@mat.upv.es

\end{document}